\documentclass[a4paper,11pt]{article}

\usepackage[T1]{fontenc}
\usepackage{lmodern}

\usepackage{dsfont}

\usepackage[latin1]{inputenc}
\usepackage{amsmath}
\usepackage{amsthm}
\usepackage{amssymb}
\usepackage{mathrsfs}
\usepackage{graphicx}
\usepackage[all]{xy}
\usepackage{hyperref}

\usepackage{stmaryrd}
\usepackage{caption}

\usepackage{abstract} 

\newtheorem{thm}{Theorem}[section]
\newtheorem{cor}[thm]{Corollary}

\newtheorem{prop}[thm]{Proposition}

\theoremstyle{definition}
\newtheorem{definition}[thm]{Definition}

\title{Groups acting on quasi-median graphs. An introduction}
\date{\today}
\author{Anthony Genevois}

\begin{document}

\maketitle

\begin{abstract}
Quasi-median graphs have been introduced by Mulder in 1980 as a generalisation of median graphs, known in geometric group theory to naturally coincide with the class of CAT(0) cube complexes. In his PhD thesis, the author showed that quasi-median graphs may be useful to study groups as well. In the present paper, we propose a gentle introduction to the theory of groups acting on quasi-median graphs. 
\end{abstract}

\tableofcontents

\section{Introduction}

\hspace{0.5cm} CAT(0) cube complexes were introduced by Gromov in his seminal paper \cite{Gromov1987} as a convenient source of examples of CAT(0) and CAT(-1) groups. But their strength really appeared with the recognition of the central role played by the combinatorics of their hyperplanes, initiated by Sageev in his thesis \cite{MR1347406}. Since then, several still open conjectures for CAT(0) spaces were verified for CAT(0) cube complexes, including the (bi)automaticity of cubulated groups \cite{NibloReeves}, the Tits Alternative for groups acting freely on finite-dimensional CAT(0) cube complexes \cite{alternative}, and the Rank Rigidity Conjecture \cite{MR2827012}. Recently, CAT(0) cube complexes were also crucial in the proof of the famous virtual Haken conjecture \cite{MR3104553}. 

Independently, Roller \cite{Roller} and Chepo\"{i} \cite{mediangraphs} realised that the class of CAT(0) cube complexes can be naturally identified with the class of the so-called \emph{median graphs}. These graphs were known by graph theorists for a long time since they were introduced by Nebesk\'y in 1971 \cite{NebeskyMedian}. Since then, several classes of graphs were introduced as generalisations of median graphs (see for instance \cite{weaklymoduloar} and references therein), including the main subject of this article, \emph{quasi-median graphs}, which were introduced by Mulder in 1980 \cite{Mulder} and more extensively studied by Bandelt, Mulder and Wilkeit in 1994 \cite{quasimedian}. 

In \cite{Qm}, we showed that quasi-median graphs appear naturally in several places in geometric group theory, and that finding group actions on such graphs may be extremely useful. The goal of this article is to give a gentle introduction to the formalism introduced in \cite{Qm}. 

The paper is organised as follows. In Section \ref{section:qmvscube}, we notice that the geometry of quasi-median graphs is quite similar to the geometry of CAT(0) cube complexes. In particular, we generalise the definition of hyperplanes and show that geometry essentially reduces to the combinatorics of the hyperplanes. In Section \ref{section:QmEx}, examples of groups acting on quasi-median graphs are given. We focus on graph products, wreath products, and diagram groups. In Section \ref{section:rotative}, we introduce and study \emph{rotative stabilisers} of hyperplanes. They are used to embed graph products into groups acting on quasi-median graphs, and, under some assumptions, a decomposition as a semidirect product is proved. Finally, Sections \ref{section:topicalI} and \ref{section:topicalII} are dedicated to \emph{topical-transitive} actions, which can be used to prove combination theorems. Applications to the groups mentioned in Section \ref{section:QmEx} are given.

\section{Quasi-median graphs look like CAT(0) cube complexes}\label{section:qmvscube}

\noindent
Quasi-median graphs may be define in many different ways; see for instance \cite{quasimedian}. In \cite{Qm}, the definition we use is the following:

\begin{definition}
A graph is \emph{weakly modular} if it satisfies the following two conditions:
\begin{description}
	\item[(triangle condition)] for any vertex $u$ and any two adjacent vertices $v,w$ at distance $k$ from $u$, there exists a common neighbor $x$ of $v,w$ at distance $k-1$ from $u$;
	\item[(quadrangle condition)] for any vertices $u,z$ at distance $k$ apart and any two neighbors $v,w$ of $z$ at distance $k-1$ from $u$, there exists a common neighbor $x$ of $v,w$ at distance $k-2$ from $u$.
\end{description}
A graph is \emph{quasi-median} if it is connected, weakly modular and if it does not contain $K_4^-$ and $K_{3,2}$ as induced subgraphs (see Figure \ref{figure7}).
\end{definition}
\begin{figure}
\begin{center}
\includegraphics[scale=0.6]{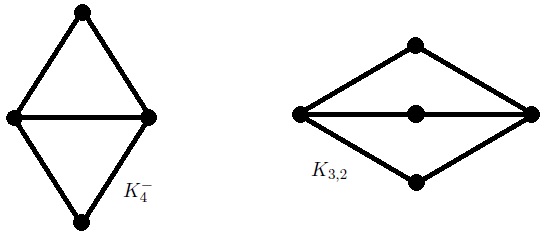}
\end{center}
\caption{The graphs $K_4^-$ and $K_{3,2}$.}
\label{figure7}
\end{figure}

\noindent
Although this definition turns out to be convenient to work with, the analogy with CAT(0) cube complexes is not clear. This analogy becomes more explicit thanks to the following two statements (which were originally proved only for finite graphs). We refer to the corresponding references for the needed definitions.  

\begin{thm}\emph{\cite{retracthypercube}}
Let $X$ be a finite graph. The following statements are equivalent:
\begin{itemize}
	\item $X$ is median;
	\item $X$ is a rectract of a hypercube;
	\item $X$ is obtained from cubes by gated amalgams.
\end{itemize}
\end{thm}

\begin{thm}\emph{\cite{WilkeitRetractsHamming, quasimedian}}
Let $X$ be a finite graph. The following statements are equivalent:
\begin{itemize}
	\item $X$ is quasi-median;
	\item $X$ is a rectract of a product of complete graphs;
	\item $X$ is obtained from prisms (ie., products of complete graphs) by gated amalgams.
\end{itemize}
\end{thm}

\noindent
Roughly speaking, if one says that CAT(0) cube complexes or median graphs are obtained by gluing cubes together in a ``nonpositively-curved way'', then quasi-median graphs are obtained by gluing prisms together in a ``nonpositively-curved way''. So edges become cliques, and cubes become prisms. This analogy motivates the following definition, which mimics the definition of hyperplanes in CAT(0) cube complexes.

\begin{definition}\label{def:hyperplanes}
A \emph{hyperplane} is an equivalence class of edges, where two edges $e$ and $e'$ are said \emph{equivalent} whenever there exists a sequence of edges $e_0=e,e_1, \ldots, e_{n-1},e_n=e'$ such that, for every $1 \leq i \leq n-1$, either $e_i$ and $e_{i+1}$ are opposite sides of some square or they are two sides of some triangle. Alternatively, if we say that two cliques are \emph{parallel} whenever they respectively contain two opposite sides of some square, then a hyperplane is the collection of edges of some class cliques with respect to the transitive closure of parallelism.
\end{definition}

\noindent
One says that an edge or a clique is \emph{dual} to a given hyperplane if it belongs to the associated class of edges. Of course, because two distinct equivalence classes are necessarily disjoint, an edge or a clique is dual to a unique hyperplane. 

\begin{definition}
The \emph{carrier} of hyperplane $J$, denoted by $N(J)$, is the subgraph generated by the union of all the edges of $J$. A \emph{fiber} of $J$ is a connected component of $\partial J = N(J) \backslash \backslash J$, ie., the subgraph obtained from $N(J)$ by removing the interiors of the edges of $J$. 
\end{definition}

\noindent
Now, the point is that, in the same way that the geometry of a CAT(0) cube complex reduces to the combinatorics of its hyperplanes, the geometry of a quasi-median graph reduces as well to the combinatorics of its hyperplanes. More precisely:

\begin{thm}\emph{\cite[Proposition 2.15 and 2.30]{Qm}}
Let $X$ be a quasi-median graph. For every hyperplane $J$, the following statements hold.
\begin{itemize}
	\item The subgraph $X \backslash \backslash J$ obtained from $X$ by removing the interiors of the edges of $J$ is disconnected, possibly with infinitely many connected components. Each such component, called a \emph{sector}, is gated. 
	\item A fiber of $J$ is a gated subgraph; in particular, it is a quasi-median graph on its own right.
	\item The carrier of $J$ is naturally isometric to a product $F \times C$ where $F$ is an arbitrary fiber of $F$ and $C$ an arbitrary clique dual to $J$. Furthermore, $N(J)$ is a gated subgraph.
\end{itemize}
Moreover, a path in $X$ is a geodesic if and only if it crosses at most once each hyperplane. As a consequence, the distance between any two vertices of $X$ coincides with the number of hyperplanes which separate them. 
\end{thm}

\noindent
Recall that, given a graph $X$ and a subgraph $Y \subset X$, a \emph{gate} in $Y$ for some vertex $x \in X$ is a vertex $y \in Y$ such that, for every $z \in Y$, there exists a geodesic between $x$ and $z$ passing through $y$. If every vertex of $X$ has a gate in $Y$, then $Y$ is \emph{gated}. It is worth noticing that a gated subgraph must be convex. In fact, gated subgraphs in quasi-median graphs play the same role as convex subcomplexes in CAT(0) cube complexes. The reason is that one can project vertices onto gated subgraphs, just by taking the corresponding gates: note that the gate, when it exists, is the unique vertex of the subgraph minimising the distance to the initial vertex. In median graphs, convex subgraphs are always gated, but it is no longer true in quasi-median graphs: an edge in a triangle is obviously convex, but the vertex of our triangle which does not belong to this edge have two nearest-point projections. For more information on projections onto gated subgraphs in quasi-median graphs, we refer to \cite[Section 2.3]{Qm}. 

\medskip \noindent
Let us mention that, in the same way that median graphs can be filled in to get a CAT(0) space, \emph{quasi-median complexes}, ie., prism complexes obtained from quasi-median graphs by filling in every clique with a simplex and every one-skeleton of an $n$-cube with an $n$-cube, define CAT(0) spaces as well.

\begin{thm}\emph{\cite[Theorem 2.120]{Qm}}
Quasi-median complexes are CAT(0).
\end{thm}

\noindent
It is worth noticing that a CAT(0) cube complex can be naturally associated to any quasi-median graph. Indeed, endow a given quasi-median graph $X$ with a structure of wallspace by declaring that a sector and its complement is wall, and next cubulate this wallspace to get a CAT(0) cube complex $C(X)$. See Figure \ref{figure1} for an example. Then the canonical map $X \to C(X)$ is a quasi-isometry, which is furthermore equivariant, meaning that if a group acts on $X$ then it naturally acts on $C(X)$ as well. As a consequence, every group acting geometrically (resp. properly, fixed-point freely) on a quasi-median graph acts geometrically (resp. properly, fixed-point freely) on a CAT(0) cube complex. 
\begin{figure}
\begin{center}
\includegraphics[scale=0.5]{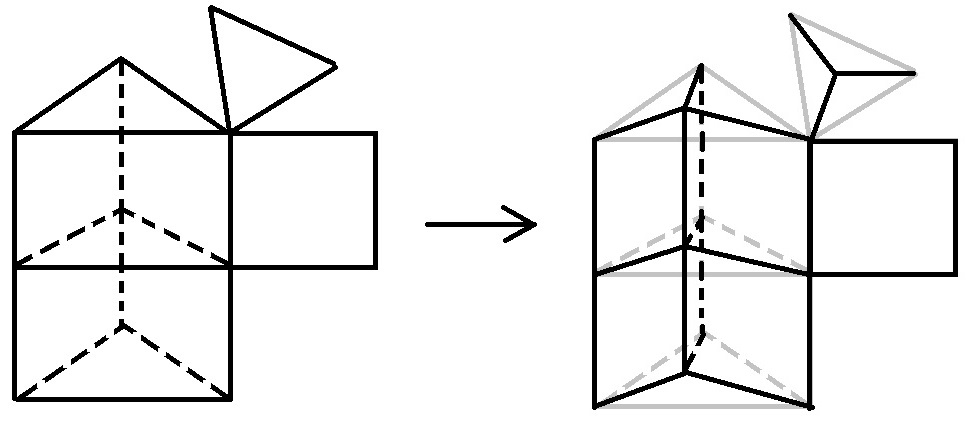}
\caption{CAT(0) cube complex associated to a quasi-median graph.}
\label{figure1}
\end{center}
\end{figure}

\noindent
Therefore, a natural question is: why do we care about quasi-median graphs? The first reason is that a quasi-median graph may appear more naturally than a CAT(0) cube complex, making the quasi-median graph easier to handle. For instance, we are convinced that quasi-median graphs provide the good framework to study the geometry of graph products (see Section \ref{section:QmEx}). The second reason is that one can exploit the specific structure of hyperplanes in quasi-median graphs to define particular kinds of actions on such graphs which provide interesting information on the group (see Sections \ref{section:rotative}, \ref{section:topicalI} and~\ref{section:topicalII}).

\section{Quasi-median graphs appear in Nature}\label{section:QmEx}

\noindent
In this section, our goal is to describe three classes of groups which act naturally on quasi-median graphs: graph products, (some) wreath products, and diagram products.

\paragraph{Graph products.} Let us begin by recalling the definition of graph products, as introduced in \cite{GreenGP}:

\begin{definition}
Let $\Gamma$ be a simplicial graph and $\mathcal{G}$ a collection of groups indexed by the vertices of $\Gamma$. The \emph{graph product} $\Gamma \mathcal{G}$ is the group defined by the relative presentation
$$\langle G_u, u \in V(\Gamma) \mid [G_u,G_v]=1, (u,v) \in E(\Gamma) \rangle,$$
where $[G_u,G_v]=1$ is an abbreviation for: $[g,h]=1$ for every $g \in G_u$ and $h \in G_v$. 
\end{definition}

\noindent
Notice that, if $\Gamma$ has no edges, then $\Gamma \mathcal{G}$ coincides with the free product of all the groups of $\mathcal{G}$; and if $\Gamma$ is complete graph, then $\Gamma \mathcal{G}$ coincides with the direct sum of all the groups of $\mathcal{G}$. Usually, one says that graph products interpolate between free products and direct sum. Next, notice also that, if all the groups of $\mathcal{G}$ are infinite cyclic, we recover right-angled Artin groups; and if all the groups of $\mathcal{G}$ are cyclic of order two, we recover the right-angled Coxeter groups. 

\begin{thm}\emph{\cite[Proposition 8.2]{Qm}}
Let $\Gamma$ be a simplicial graph and $\mathcal{G}$ a collection of groups indexed by $V(\Gamma)$. The Cayley graph $X(\Gamma, \mathcal{G})$ of $\Gamma \mathcal{G}$ associated to the generating set $\bigsqcup\limits_{G \in \mathcal{G}} G \backslash \{1 \}$ is a quasi-median graph.
\end{thm}

\noindent
Graph products are our examples where the link between groups and graphs is the strongest. In \cite{Qm}, one of our major contributions to the study of geometric properties of graph products is the characterisation of relatively hyperbolic graph products. Before stating our theorem, we need to introduce some vocabulary.

\medskip \noindent
Given a finite simplicial graph $\Gamma$ and a collection of groups $\mathcal{G}$ indexed by $V(\Gamma)$, we will say that a subgraph $\Lambda \leq \Gamma$ is \emph{vast} if the subgroup of $\Gamma \mathcal{G}$ generated by the vertex-groups corresponding to the vertices of $\Lambda$, ie., $\Lambda \mathcal{G}$, is infinite; otherwise, $\Lambda$ is said \emph{narrow}. Notice that a subgraph is narrow if and only if it is complete and all the vertex-groups labelling its vertices are finite. A join $\Lambda_1 \ast \Lambda_2 \leq \Gamma$ is \emph{large} if both $\Lambda_1$ and $\Lambda_2$ are vast. 

\begin{definition}
Let $\Gamma$ be a finite simplicial graph and $\mathcal{G}$ a collection of groups labelled by $V(\Gamma)$. For every subgraph $\Lambda \subset \Gamma$, let $\mathrm{cp}(\Lambda)$ denote the subgraph of $\Gamma$ generated by $\Lambda$ and the vertices $v \in \Gamma$ such that $\mathrm{link}(v) \cap \Lambda$ is vast. Now, define the collection of subgraphs $\mathfrak{J}^n(\Gamma)$ of $\Gamma$ by induction in the following way:
\begin{itemize}
	\item $\mathfrak{J}^0(\Gamma)$ is the collection of all the large joins in $\Gamma$;
	\item if $C_1, \ldots, C_k$ denote the connected components of the graph whose set of vertices is $\mathfrak{J}^n(\Gamma)$ and whose edges link two subgraphs with vast intersection, we set $\mathfrak{J}^{n+1}(\Gamma) = \left( \mathrm{cp} \left( \bigcup\limits_{\Lambda \in C_1} \Lambda \right), \ldots, \mathrm{cp} \left( \bigcup\limits_{\Lambda \in C_k} \Lambda \right) \right)$.
\end{itemize}
Because $\Gamma$ is finite, the sequence $(\mathfrak{J}^n(\Gamma))$ must eventually be constant and equal to some collection $\mathfrak{J}^{\infty}(\Gamma)$. Finally, let $\mathfrak{J}(\Gamma)$ denote the collection of subgraphs of $\Gamma$ obtained from $\mathfrak{J}^{\infty}(\Gamma)$ by adding the singletons corresponding to the vertices of $\Gamma \backslash \bigcup \mathfrak{J}^{\infty}(\Gamma)$.  
\end{definition}

\begin{thm}\emph{\cite[Theorem 8.35]{Qm}}
Let $\Gamma$ be a finite simplicial graph not reduced to a single vertex and $\mathcal{G}$ a collection of finitely generated groups labelled by $V(\Gamma)$. The graph product $\Gamma \mathcal{G}$ is relatively hyperbolic if and only if $\mathfrak{J}(\Gamma) \neq \{ \Gamma \}$. If so, $\Gamma \mathcal{G}$ is hyperbolic relatively to $\{ \Lambda \mathcal{G} \mid \Lambda \in \mathfrak{J}(\Gamma) \}$. 
\end{thm}

\paragraph{Wreath products.} Recall that, given two groups $G$ and $H$, the \emph{wreath product} $G \wr H$ is the semidirect product $\left( \bigoplus\limits_{h \in H} G \right) \rtimes H$ where $H$ acts on the direct sum by permuting the coordinates. Often, such groups are described as \emph{lamplighter groups} in the following way. Fix two generating sets $R$ and $S$ of $G$ and $H$ respectively. Notice that $R \cup S$, when we identify $G$ is its copy in the direct sum labelled by the identity element, generates $G \wr H$. An element $\left( (g_h)_{h \in H}, k \right)$ of the wreath product is thought of as the following configuration: each vertex $h \in H$ of the Cayley graph of $H$ (corresponding to $S$) is a lamp whose color is $g_h \in G$, and an arrow points to the vertex $k \in H$ to indicate the position of the lamplighter. Now, right-multiplicating this element by some $r \in R$ corresponds to modifying the color of the lamp at the vertex $k \in H$, ie., where the lamplighter is, from $g_k$ to $g_kr$; and right-multiplicating the element by some $s \in S$ corresponds to moving the lamplighter from $k$ to the adjacent vertex $ks$. Thus, the Cayley graph of $G \wr H$ (corresponding the generating set $R \cup S$) encodes a lamplighter moving along the Cayley graph of $H$ and modifying the colors of the lamps along its path. 

\medskip \noindent
Now, suppose that the group $H$ acts on a CAT(0) cube complex $X$, and suppose without loss of generality that $X$ contains a vertex $x_0 \in X$ with trivial stabiliser. In the previous description, we will replace the Cayley graph of $H$ (ie., the graph on which the lamplighter moves) with the cube complex $X$; the lamps will be the vertices of the orbit $\Omega = H \cdot x_0 \subset X$; and the lamplighter (ie., the arrow) will be replaced with a non empty finite convex subcomplex of $X$. Formally, the situation is the following:

\begin{definition}
A \emph{wreath} $(C, \varphi)$ is the data of a non empty finite convex subcomplex $C \subset X$ and a function $\varphi : \Omega \to G$ with finite support, ie., $\varphi(p)=1$ for all but finitely many $p \in \Omega$. The \emph{graph of wreaths} $\mathfrak{W}$ is the graph whose vertices are the wreaths and whose edges link two wreaths $(C, \varphi)$ and $(Q, \psi)$ either if $\varphi= \psi$ and if there exists a unique hyperplane intersecting exactly one of $C$ and $Q$; or if $C=Q$ and if $\varphi$ and $\psi$ differ on a single point of $C \cap \Omega =Q \cap \Omega$. 
\end{definition} 

\noindent
A funny interpretation is the following. By replacing the arrow by a subcomplex, the lamplighter becomes ``quantum''. It has no precise position, it is everywhere inside the subcomplex, which can be thought of as a ``cloud'', and it can modify the color of any lamp inside this cloud (but just one at each time). The cloud moves by adding or removing a hyperplane from the corresponding subcomplex, the operation being allowed only if the resulting subcomplex remains convex and non empty. 

\begin{thm}\emph{\cite[Proposition 9.10]{Qm}}
The graph of wreaths $\mathfrak{W}$ is quasi-median. 
\end{thm}

\noindent
(We used the same idea in \cite{aTmW} to construct actions of wreath products on median spaces, reproving that acting properly on a CAT(0) cube complex and being a-T-menable are stable under wreath products.)

\paragraph{Diagram products.} In \cite{MR1725439}, Guba and Sapir introduced \emph{diagram products} from the class of \emph{diagram groups} \cite{MR1396957}, in a similar way that graph products can be derived from right-angled Artin groups. 

\begin{definition}
Let $\mathcal{P}= \langle \Sigma \mid \mathcal{R} \rangle$ be a semigroup presentation. We suppose that, for every relation $u=v \in \mathcal{R}$, the relation $v=u$ does not belong to $\mathcal{R}$; as a consequence, $\mathcal{R}$ does not contain relations of the form $u=u$. Let $S(\mathcal{P})$ denote the square complex 
\begin{itemize}
	\item whose vertices are the words written over the alphabet $\Sigma$;
	\item whose edges have the form $(a,u \to v,b)$ for some relation $u=v \in \mathcal{R}$, linking the vertices $aub$ and $avb$;
	\item whose squares have the form $(a,u \to v,b,p \to q,c)$ for some relations $u=v,p=q \in \mathcal{R}$, linking the vertices $aubpc$, $avbpc$, $aubqc$ and $avbqc$.
\end{itemize}
Now, fix a collection of groups $\mathcal{G}$ indexed by the alphabet $\Sigma$, and let $S(\mathcal{P}, \mathcal{G})$ be the complex of groups (we refer to \cite[Chapter II.12]{MR1744486} for more information on (simple) complexes of groups) such that
\begin{itemize}
	\item the underlying complex is $S(\mathcal{P})$;
	\item the vertex-group associated to the word $s=s_1 \cdots s_n$ is $G_s= G_{s_1} \times \cdots \times G_{s_n}$;
	\item the edge-group associated to the edge $(a,u \to v,b)$ is $G_a \times G_b$;
	\item the square-groups are all trivial;
	\item the embedding of an edge-group into a vertex-group coincides the canonical embedding of a factor into a direct product.
\end{itemize}
Fixing some baseword $w \in \Sigma^+$, the \emph{diagram product} $D(\mathcal{P}, \mathcal{G},w)$ is the fundamental group of the complex of groups $S(\mathcal{P}, \mathcal{G},w)$ corresponding to the connected component of $S(\mathcal{P}, \mathcal{G})$ containing the vertex $w$. 
\end{definition}

\noindent
When all the groups of $\mathcal{G}$ are trivial, one recovers the definition of diagram groups as fundamental groups of Squier complexes \cite{MR1396957}.

\medskip \noindent
Thanks to the alternative description of diagram groups given in \cite[Section 10]{Qm}, we generalised the definition in \cite{MoiV} by looking at \emph{braided diagram groups} \cite{MR1396957, FarleyPicture}. It would be too long to define this generalisation here, but the point is that there exist a \emph{braided diagram product} $D_b(\mathcal{P}, \mathcal{G},w)$ and an \emph{annular diagram product} $D_a(\mathcal{P}, \mathcal{G}, w)$, such that
$$D(\mathcal{P},\mathcal{G}, w) \subset D_a(\mathcal{P},\mathcal{G}, w) \subset D_b(\mathcal{P}, \mathcal{G}, w).$$
To avoid ambiguity, the diagram product $D(\mathcal{P}, \mathcal{G},w)$ will be sometimes referred to as the \emph{planar diagram product}. Each of these products acts on quasi-median graphs, denoted by $X(\mathcal{P}, \mathcal{G},w)$, $X_a(\mathcal{P}, \mathcal{G},w)$ and $X_b(\mathcal{P}, \mathcal{G},w)$ respectively \cite{Qm, MoiV}. For instance, the graph $X(\mathcal{P}, \mathcal{G},w)$, roughly speaking, coincides with a connected component of the natural Cayley graph of the fundamental groupoid corresponding to the complex of groups $S(\mathcal{P}, \mathcal{G},w)$.

\section{Rotative actions on quasi-median graphs}\label{section:rotative}

\noindent
The main difference between CAT(0) cube complexes and quasi-median graphs is that hyperplanes may separate the graph into more than two pieces. In fact, in the cases which interest us, the number of sectors delimited by a given hyperplane turns out to be infinite. Therefore, an isometry of infinite order stabilising a fixed hyperplane may only ``rotate'' this hyperplane without ``translating'' vertices. More formally:

\begin{definition}
Let $G$ be a group acting on a quasi-median graph $X$ and $J$ a hyperplane of $X$. The \emph{rotative stabiliser} of $J$ is
$$\mathrm{stab}_{\circlearrowleft}(J)= \bigcap \{ \mathrm{stab}(C) \mid \text{$C$ clique of $J$} \}.$$
\end{definition}

\noindent
An interesting remark is that the rotative stabilisers of two transverse hyperplanes commute, ie., any element of one stabiliser commutes with any element of the other. Loosely speaking, the situation is similar to two rotations in the space whose axes are orthogonal. On the other hand, it is not difficult to show that the group of isometries generated by the rotative stabilisers of two non transverse hyperplanes decomposes as the free product of these two rotative stabilisers (provided that vertex-stabilisers are trivial). Essentially, it is sufficient to play ping-pong with the sectors delimited by the hyperplanes. Generalising the argument to an arbitrary collection of hyperplanes leads to the following statement:

\begin{thm}\label{thm:embedGP}\emph{\cite[Theorem 8.43]{Qm}}
Let $G$ be a group acting on a quasi-median graph $X$ with trivial vertex-stabilisers. For every hyperplane $J$, choose a residually finite subgroup $H_J$ of its rotative stabiliser. If $\Gamma$ is an induced finite subgraph of the crossing graph of $X$, there exists a collection $\mathcal{G}$ of finite-index subgroups of our $H_J$'s such that the graph product $\Gamma \mathcal{G}$ embeds into $G$. 
\end{thm}

\noindent
In a quasi-median graph $X$, the \emph{crossing graph} of $X$ is the graph, denoted by $\Delta X$, whose vertices are the hyperplanes of $X$ and whose edges link two transverse hyperplanes. An interesting particular case of the previous statement is when the $H_J$'s are all infinite cyclic.

\begin{cor}
Let $G$ be a group acting on a quasi-median graph $X$ with trivial vertex-stabilisers and with infinite non torsion rotative stabilisers. If $\Gamma$ is a finite induced subgraph in the crossing graph of $X$, then the right-angled Artin group $A(\Gamma)$ embeds into $G$.
\end{cor}

\noindent
In general, describing the crossing graph of a quasi-median graph is difficult, but for graph products, one gets the following nice description. Let $\Gamma$ be a simplicial graph and $\mathcal{G}$ a collection of non trivial groups indexed by $V(\Gamma)$. The crossing graph of the quasi-median graph $X(\Gamma , \mathcal{G})$ coincides with the graph 
\begin{itemize}
	\item whose vertices are the conjugates of vertex-groups $gG_ug^{-1}$;
	\item whose edges link two conjugates $gG_ug^{-1}$ and $hG_vh^{-1}$ if they commute, ie., any element of one conjugate commutes with any element of the other.
\end{itemize}
Therefore, the crossing graph $\Delta X(\Gamma, \mathcal{G})$ naturally generalises the \emph{extension graph} defined for right-angled Artin groups in \cite{embeddingRAAG}, and Theorem \ref{thm:embedGP} generalises \cite[Theorem~2]{embeddingRAAG}. 

\medskip \noindent
Pushing further the argument used to prove Theorem \ref{thm:embedGP}, it is possible to find a structure theorem for groups acting on quasi-median graphs with ``sufficiently large'' rotative stabilisers.

\begin{definition}
Let $G$ be a group acting on a quasi-median graph $X$ and $\mathcal{J}$ a collection of hyperplanes. The action $G \curvearrowright X$ is \emph{$\mathcal{J}$-rotative} if, for every $J \in \mathcal{J}$, the rotative stabiliser $\mathrm{stab}_{\circlearrowleft}(J)$ acts transitively and freely on the set of sectors delimited by $J$. 
\end{definition}

\noindent
Before stating our theorem, we also need the following definition: given a quasi-median graph $X$, a collection of hyperplanes $\mathcal{J}$ and a base vertex $x_0 \in X$, one says that a subcollection $\mathcal{J}_0 \subset \mathcal{J}$ is \emph{$x_0$-peripheral} if, for every $J \in \mathcal{J}_0$, there does not exist a hyperplane of $\mathcal{J}$ separating $J$ and $x_0$. 

\begin{thm}\label{thm:rotativeactions}\emph{\cite[Theorem 10.54]{Qm}}
Let $G$ be a group acting $\mathcal{J}$-rotatively on a quasi-median graph $X$. Fix a basepoint $x_0 \in X$. If $Y \subset X$ denotes the intersection of the sectors containing $x_0$ which are delimited by a hyperplane of $\mathcal{J}$, then
$$G= \mathrm{Rot}(\mathcal{J}) \rtimes \mathrm{stab}(Y), \ \text{where} \ \mathrm{Rot}(\mathcal{J})= \langle  \mathrm{stab}_{\circlearrowleft}(J), \ J \in \mathcal{J} \rangle.$$
Moreover, if $\mathcal{J}_0 \subset \mathcal{J}$ denotes the unique maximal $x_0$-peripheral subcollection of $\mathcal{J}$, then $\mathrm{Rot}(\mathcal{J})$ decomposes as a graph product $\Delta \mathcal{G}$, where $\Delta$ is the graph whose vertices are the hyperplanes of $\mathcal{J}_0$ and whose edges link two hyperplanes which are transverse, and where $\mathcal{G}= \{ \mathrm{stab}_{\circlearrowleft}(J) \mid J \in \mathcal{J}_0 \}$.
\end{thm}

\noindent
As an application, we showed in \cite[Theorem 10.58]{Qm} that (planar) diagram products decompose as semidirect products between a graph product and their underlying diagram groups. By taking collections of infinite cyclic groups and trivial diagram groups, it implies that many right-angled Artin groups turn out to be diagram groups as well \cite[Corollary 10.60]{Qm}. In \cite{MoiV}, we generalised this argument and proved the following statement (see also \cite[Theorem 5.7]{MR2193191} for planar diagram groups):

\begin{thm}
Let $D$ be a planar (resp. annular, braided) diagram group. There exist a subgroup $R$ of some right-angled Artin group and a subgroup $S$ of Thompson's group $F$ (resp. $T$, $V$), such that $D$ decomposes as the short exact sequence
$$1 \to R \to D \to S \to 1.$$
\end{thm}

\noindent
As a consequence, any simple diagram group embeds into the corresponding Thompson's group. Our motivation in \cite{MoiV} was to construct new subgroups of Thompson's group~$V$.

\section{Topical-transitive actions on quasi-median graphs I}\label{section:topicalI}

\noindent
An interesting point is that, in most of the examples of groups acting on quasi-median graphs that we know, combination theorems can be proved with respect to clique-stabilisers. Let us illustrate such arguments with the following statement.

\begin{thm}
Let $\Gamma$ be a finite simplicial graph and $\mathcal{G}$ a collection of groups indexed by the vertices of $\Gamma$. Suppose that every group of $\mathcal{G}$ acts metrically properly on a CAT(0) cube complex. Then so does the graph product $\Gamma \mathcal{G}$. 
\end{thm}

\begin{proof}[Sketch of proof.]
For every $G \in \mathcal{G}$, fix a CAT(0) cube complex $X(G)$ on which it acts metrically properly; via the orbit map
$$\mathcal{O} : \left\{ \begin{array}{ccc} G & \to & X(G) \\ g & \mapsto & g \cdot x_0(G) \end{array} \right.,$$
where $x_0(G) \in X(G)$ is a fixed vertex, one may naturally endow $G$ with a structure of wallspace, ie., set
$$\mathcal{W}(G)= \left\{ \{ \mathcal{O}^{-1}(D), \mathcal{O}^{-1}(D^c) \}, \ D \ \text{halfspace of $X$} \right\}.$$
Now, fix a clique $C$ of the quasi-median graph $X(\Gamma, \mathcal{G})$ on which $\Gamma \mathcal{G}$ acts. Recall that $X(\Gamma, \mathcal{G})$ is the Cayley graph of $\Gamma \mathcal{G}$ with respect to the generating set $\bigcup\limits_{G \in \mathcal{G}} G \backslash \{ 1 \}$. As a consequence, the cliques of $X(\Gamma, \mathcal{G})$ coincide with cosets of vertex-groups. Therefore, there exist some $g \in \Gamma \mathcal{G}$ and some $G \in \mathcal{G}$ such that $C=gG$. The walls defined on $G$ naturally define walls on the clique $C$ (notice that this transfer does not depend on our choice of $g$, which is used to identify $C$ with $G$, since $g$ is uniquely determined up to right-multiplication by an element of $G$  and that $\mathcal{W}(G)$ is $G$-invariant), and we may extend them as walls on the whole quasi-median graph by setting
$$\overline{\mathcal{W}}(C) = \left\{ \{ \mathrm{proj}_C^{-1}(D), \mathrm{proj}^{-1}_C(D^c) \}, \ \{D,D^C \} \in \mathcal{W}(C) \right\},$$
where $\mathcal{W}(C)$ denotes the collection of walls on $C$ induced by the walls of $G$. Thus, one gets a collection of walls
$$\bigcup\limits_{\text{$C$ clique}} \overline{\mathcal{W}}(C)$$
on $X(\Gamma, \mathcal{G})$, which is $\Gamma \mathcal{G}$-invariant by construction. However, in general infinitely many walls separate two given vertices of our graph. Nevertheless, it turns out that, if one removes duplicated walls (ie., walls inducing the same partition of $X(\Gamma, \mathcal{G})$), then our collection of walls endows $X(\Gamma, \mathcal{G})$ with a structure of wallspace. The key observation is that the system of wallspaces $\{ \mathcal{W}(C) \mid \text{$C$ clique} \}$ is \emph{coherent}, meaning that, for every clique $C$ and $C'$ dual to the same hyperplane, the equality
$$t_{C \to C'} \mathcal{W}(C) = \mathcal{W}(C')$$
holds, where $t_{C \to C'}$ denotes the restriction to $C$ of the projection onto $C'$ (this map may be thought of as a ``parallel transport'' along the hyperplane dual to $C$ and $C'$). Indeed, because $C$ and $C'$ are dual to the same hyperplane of $X(\Gamma, \mathcal{G})$, there exist $g \in \Gamma \mathcal{G}$, $G \in \mathcal{G}$ and $h \in C_{\Gamma \mathcal{G}} (G)$ such that $C=gG$ and $C'=ghG$, so that
$$\mathcal{W}(C')= \mathcal{W}(ghG) = gh \mathcal{W}(G) = g \mathcal{W}(G)h = \mathcal{W}(C)h = t_{C \to C'} \mathcal{W}(C),$$
since the right-multiplication by $h$ turns out to coincide with the map $t_{gG \to ghG}$. As a consequence, the collections of walls $\overline{\mathcal{W}}(C)$ and $\overline{\mathcal{W}}(C')$ coincide if $C$ and $C'$ are two cliques dual to the same hyperplane $J$. So it makes sense to define $\mathcal{W}(J)$ as the collection $\overline{\mathcal{W}}(C)$ for any clique $C$ dual to $J$, and we finally endow $X(\Gamma, \mathcal{G})$ with the collection of walls
$$\mathcal{HW} = \bigcup\limits_{\text{$J$ hyperplane}} \mathcal{W}(J).$$
It can be proved that, if we write an element $g \in \Gamma \mathcal{G}$ as a word of minimal length $g_1 \cdots g_n$, where $g_i$ belongs to some $G_i \in \mathcal{G}$ for every $1 \leq i \leq n$, then
$$d_{\mathcal{HW}}(1,g) = \sum\limits_{i=1}^n d_{\mathcal{W}(G_i)}(1,g_i) = \sum\limits_{i=1}^n d_{X(G_i)} (x_0(G_i),g_i \cdot x_0(G_i)),$$
where $d_{\mathcal{HW}}(\cdot, \cdot)$ denotes the number of walls in $\mathcal{HW}$ separating two vertices of $X(\Gamma, \mathcal{G})$, and, for every $G \in \mathcal{G}$, $d_{\mathcal{W}(G)}( \cdot, \cdot)$ the number of walls of $\mathcal{W}(G)$ separating two points of $G$ and $d_{X_0(G)}$ the distance in the CAT(0) cube complex $X_0(G)$. From this formula, it follows that only finitely many walls of $\mathcal{HW}$ separate two given vertices of $X(\Gamma, \mathcal{G})$, so that $(X(\Gamma, \mathcal{G}), \mathcal{HW})$ defines a wallspace; and that the action of $\Gamma \mathcal{G}$ on the CAT(0) cube complex obtained by cubulating this wallspace is metrically proper. 
\end{proof}

\noindent
In the previous argument, two points are fundamental:
\begin{itemize}
	\item For every clique $C$ of our quasi-median graph, the action $\mathrm{stab}(C) \curvearrowright C$ is free and transitive on the vertices. This allows us to identify a clique with its stabiliser in order to transfer some structure from $\mathrm{stab}(C)$ to $C$ (in the previous proof, a collection of walls).
	\item There exists some compatibility between the action $\Gamma \mathcal{G} \curvearrowright X(\Gamma, \mathcal{G})$ and the maps $t_{C \to C'}$, which implies that the collection of structures (in the previous proof, collections of walls) defined on each clique is \emph{coherent}. 
\end{itemize}
In \cite{Qm}, we introduced \emph{topical-transitive} actions to recover these properties and to generalise the previous argument to other groups acting on quasi-median graphs.

\begin{definition}
Let $G$ be a group acting on a quasi-median graph $X$. The action is \emph{topical} if, for every hyperplane $J$, every clique $C$ dual to $J$ and every element $g \in \mathrm{stab}(J)$, there exists some $\rho_C(g) \in \mathrm{stab}(C)$ such that $g$ and $\rho_C(g)$ induce the same permutation on the set of sectors delimited by $J$. If moreover the action $\mathrm{stab}(C) \curvearrowright C$ is free and transitive on the vertices for every clique $C$ which either is infinite or satisfies $\mathrm{stab}(C) \neq \mathrm{fix}(C)$, the action $G \curvearrowright X$ is \emph{topical-transitive}. 
\end{definition}

\noindent
Fix a group $G$ acting topically-transitively on some quasi-median graph $X$. For convenience, we suppose that the action $\mathrm{stab}(C) \curvearrowright C$ is free and transitive on the vertices for every clique $C$ of $X$ (which happens, for instance, if all the cliques of $X$ are infinite). Fixing a collection of cliques $\mathcal{C}$ such that every orbit of hyperplane intersects $\mathcal{C}$ along exactly one clique ($\mathcal{C}$ can be thought of as a collection of cliques of reference), the key point is that, for every clique $C$ labelled by $Q \in \mathcal{C}$ (ie., some translate of $C$ is dual to the same hyperplane as $Q$), there exist a natural bijection $\phi_C : C \to Q$ and a natural isomorphism $\varphi_C : \mathrm{stab}(C) \to \mathrm{stab}(Q)$, so that
\begin{itemize}
	\item $\phi_{C} = \phi_{C'} \circ t_{C \to C'}$ for every clique $C$ and $C'$ dual to the same hyperplane;
	\item $\phi_{gC}(gx) = \varphi_{C}(s_{C}(g)) \cdot \phi_C(x)$ for every clique $C$, every vertex $x \in C$ and every element $g \in G$.
\end{itemize}
The map $s_C : G \to \mathrm{stab}(C)$ is defined below, but the point to keep in mind is that $s_C(g)=g$ for every $g \in \mathrm{stab}(C)$, so that $\phi_C$ turns out to be a $\varphi_C$-equivariant bijection. 

\medskip \noindent
Before describing the maps $\phi_C$ and $\varphi_C$, let us show how to create \emph{invariant} and \emph{coherent systems of structures}, as in the argument above. We illustrate the construction only for wallspaces, but the same idea can be applied to measured wallspaces, topologies, $\sigma$-algebras, and to collections of maps such as metrics, embeddings into Banach or Hilbert spaces, and so on. 

\begin{prop}\label{prop:GPcube}
Let $G$ be a group acting topically-transitively on some quasi-median graph $X$. Suppose that every vertex of $X$ belongs to finitely many cliques and that vertex-stabilisers are finite.  If clique-stabilisers act metrically properly on CAT(0) cube complexes, then so does $G$.
\end{prop}

\begin{proof}[Sketch of proof.]
Fix a collection of cliques $\mathcal{C}$ intersecting each orbit of hyperplanes along exactly one clique. For every $Q \in \mathcal{C}$, let $X(Q)$ be a CAT(0) cube complex on which $\mathrm{stab}(Q)$ acts metrically properly. As above, use an orbit map to endow $\mathrm{stab}(Q)$ with a $\mathrm{stab}(Q)$-invariant structure of wallspace $\mathcal{W}(Q)$. If $C$ is an arbitrary clique of $X$, labelled by $Q \in \mathcal{C}$ (ie., $C$ can be translated by an element of $G$ in the hyperplane dual to $Q$), set
$$\mathcal{W}(C)= \phi_C^{-1} \mathcal{W}(Q).$$
Notice that, for every clique $C$ and every element $g \in G$,
$$g \mathcal{W}(C)= g \phi_{C}^{-1} \mathcal{W}(Q)= \phi_{gC}^{-1} \left( \varphi_C(g) \cdot \mathcal{W}(Q) \right) =\phi_{gC}^{-1} \mathcal{W}(Q)=  \mathcal{W}(gC).$$
Therefore, our system of wallspaces is $G$-invariant. Moreover, if $C$ and $C'$ are two cliques dual to the same hyperplane, then
$$\mathcal{W}(C')= \phi_{C'}^{-1} \mathcal{W}(Q) = t_{C \to C'} \circ \phi_{C}^{-1} \mathcal{W}(Q)= t_{C \to C'} \mathcal{W}(C).$$
So our system of wallspaces is also coherent. As a consequence, it make sense to set, for every hyperplane $J$ of $X$, 
$$\mathcal{W}(J) = \left\{ \{ \mathrm{proj}_C^{-1}(D), \mathrm{proj}_C^{-1}(D^c) \}, \ \{D,D^c \} \in \mathcal{W}(C) \right\}$$
for some clique $C$ dual to $J$. We can show that
$$\mathcal{HW}= \bigcup\limits_{\text{$J$ hyperplane}} \mathcal{W}(J)$$
endow $X$ with a $G$-invariant structure of wallspace, so that one finds an action of $G$ on the associated CAT(0) cube complex. Under the assumptions of our proposition, it can be proved that this action is metrically proper. 
\end{proof}

\noindent
Now, we focus on the maps $\phi_C$ and $\varphi_C$. Recall that the setting is the following: a group $G$ acts topically-transitively on a quasi-median graph $X$, we fix a collection of cliques $\mathcal{C}$ intersecting each orbit of hyperplanes along a single clique, and we suppose that every clique-stabiliser acts on its clique freely and transitively on the vertices. We also fix a vertex $x_0 \in X$, and, for every clique $C$, we denote by $x_0(C)$ the projection of $x_0$ onto $C$. 

\begin{definition}
Let $C$ be a clique and $g \in G$. Denote by $p_C(g)$ the unique element of $g \cdot \mathrm{stab}(C)$ sending $x_0(C)$ to $x_0(gC)$ (such an element existing since $\mathrm{stab}(C) \curvearrowright C$ is transitive on the vertices). Also, set $s_C(g)= p_C(g)^{-1}g$.
\end{definition}

\noindent
The picture to keep in mind is the following. We decompose $g$ as a prefix $p_C(g)$ and a suffix $s_C(g)$ such that $s_C(g) \in \mathrm{stab}(C)$ ``rotates'' $C$ and such that $p_C(g)$ ``translate'' $C$ to $gC$. 

\begin{definition}
Let $C$ be a clique labelled by $Q \in \mathcal{C}$. Fix an element $g \in G$ satisfying $p_C(g)=g$ and such that $gC$ is dual to the same hyperplane as $Q$. We define
$$\phi_C : \left\{ \begin{array}{ccc} C & \to & Q \\ x & \mapsto & t_{gC \to Q}(gx) \end{array} \right. \ \text{and} \ \varphi_C : \left\{ \begin{array}{ccc} \mathrm{stab}(C) & \to & \mathrm{stab}(Q) \\ h & \mapsto & \rho_Q(ghg^{-1}) \end{array} \right. .$$
\end{definition}

\noindent
Roughly speaking, $\phi_C$ translate $C$ into the hyperplane dual to $Q$ (by left-multiplicating by $g$) and next translate $gC$ along this hyperplane to $Q$ (thanks to the map $t_{gC \to Q}$); and $\varphi_C$ sends an element $h \in \mathrm{stab}(C)$ to the unique element of $\mathrm{stab}(Q)$ which induces the same permutation on the set of sectors delimited by $J$ as $ghg^{-1} \in \mathrm{stab}(gC)$. It is worth noticing that these maps does not depend on our choice of $g$ \cite[Claims 5.30 and 5.31]{Qm}. 

\paragraph{Applications.} The argument described in the proof of Proposition \ref{prop:GPcube} can be applied to many different contexts. In \cite{Qm}, we state general combinatorial theorems about:
\begin{itemize}
	\item relative hyperbolicity \cite[Theorem 5.17]{Qm};
	\item metrically proper and geometric actions on CAT(0) cube complexes \cite[Propositions 5.22 and 5.23]{Qm};
	\item a-T-menability and a-$L^p$-menability \cite[Propositions 5.25 and 5.26]{Qm};
	\item (equivariant) $\ell^p$-compressions \cite[Proposition 5.37]{Qm}.
\end{itemize}
As examples of concrete applications:
\begin{itemize}
	\item we showed that a graph product (along a finite simplicial graph) of groups acting geometrically on CAT(0) cube complexes acts geometrically on a CAT(0) cube complex \cite[Theorem 8.17]{Qm};
	\item we determined precisely when a graph product is relatively hyperbolic \cite[Theorem 8.35]{Qm};
	\item we showed that acting metrically properly on some CAT(0) cube complex is stable under wreath products \cite[Theorem 9.28]{Qm} (see also \cite{CornulierCommensurated, aTmW});
	\item we computed equivariant $\ell^p$-compressions of some wreath products \cite[Theorem 9.37]{Qm}.
\end{itemize}
 A particular case of the last point is the following statement (in which $\alpha_p^*(\cdot)$ denotes the equivariant $\ell^p$-compression):

\begin{thm}\emph{\cite[Theorem 9.54]{Qm}}
Let $H$ be a hyperbolic group acting geometrically on some CAT(0) cube complex. For every finitely generated group $G$ and every $p \geq 1$,
$$\alpha_p^*(G \wr H ) \geq \min \left( \frac{1}{p} , \alpha_p^*(G) \right),$$
with equality if $H$ is non elementary and $p \in [1,2]$. 
\end{thm}

\section{Topical-transitive actions on quasi-median graphs II}\label{section:topicalII}

\noindent
Given a group acting topically-transitively on a quasi-median graph, the strategy described in the previous section was to transfer structures from clique-stabilisers to cliques, to extend the collection of ``local'' structures to a ``global'' structure which is invariant under the group action, and finally to exploit this action to deduce information on the group. However, many interesting group properties cannot be expressed on the group itself. For instance, being CAT(0), ie., acting geometrically on some CAT(0) space. In \cite[Sections 6 and 7]{Qm}, we showed how to modify the first step of our strategy to avoid the difficulty. For instance, if we fix a group $G$ acting topically-transitively on some quasi-median graph $X$ and if we suppose that clique-stabilisers act on CAT(0) spaces, our strategy to construct a CAT(0) space on which $G$ acts is the following (once again, we suppose for simplicity that each clique-stabiliser acts freely and transitively on the vertices of its clique):
\begin{itemize}
	\item Fix a collection of cliques $\mathcal{C}$ intersecting each orbit of hyperplanes along a single clique, and, for every $C \in \mathcal{C}$, fix a CAT(0) space $Y(C)$ on which $\mathrm{stab}(C)$ acts. Without loss of generality, we may suppose that $Y(C)$ contains a point $y_0(C)$ whose stabiliser is trivial. As above, we also fix a basepoint $x_0 \in X$, and, for every clique $C$, we denote by $x_0(C)$ the projection of $x_0$ onto $C$. Using the map $$\left\{ \begin{array}{ccc} C & \to & Y(C) \\ g \cdot x_0(C) & \mapsto & g \cdot y_0(C) \end{array} \right.,$$
we identify each clique $C \in \mathcal{C}$ with a subspace of $Y(C)$. 
	\item Next, we ``add the missing vertices'' to the clique $C$ to get a copy of $Y(C)$. This operation is called \emph{inflating the hyperplanes} of $X$. The point is that we get a new quasi-median graph $Y$, containing $X$ as an isometrically embedded subgraph, such that each clique $C_+$, which is the unique clique of $Y$ containing a given clique $C \in \mathcal{C}$, has its vertices in bijection with $Y(C)$. Moreover, if $\mathcal{C}_+$ denotes the collection of the $C_+$'s, then the action $G \curvearrowright X$ extends to a $\mathcal{C}_+$-topical action $G \curvearrowright Y$. (However, the action is no longer topical-transitive since the actions $\mathrm{stab}(C) \curvearrowright Y(C)$ are generally not transitive.)
	\item So far the situation is the following. Our group $G$ acts $\mathcal{C}_+$-topically on the quasi-median graph $Y$, and each clique $C \in \mathcal{C}_+$ can be naturally identified with the CAT(0) space $Y(C)$. Thus, now we can endow each clique $C \in \mathcal{C}_+$ with the CAT(0) metric of $Y(C)$. The next step would be to use the maps $\phi_C$ and $\varphi_C$ to construct a coherent and $G$-invariant system of CAT(0) metrics, but as the action is not topical-transitive these maps are no longer well-defined. However, it is possible to mimic the definitions of $\phi_C$ and $\varphi_C$ in more general contexts to construct maps with similar properties \cite[Sections 5.1 and 5.2]{Qm}, so that we are able to extend our collection of CAT(0) metrics to a coherent and $G$-invariant system of CAT(0) metrics $\{ (C,\delta_C) \mid \text{$C$ clique}\}$. (The point is that our new maps $\phi_C$ and $\varphi_C$ are not canonical: they depend on choices of some elements of $G$. However, according to \cite[Theorem 5.1]{Qm}, the system of metrics which is obtained does not depend on these choices.)
	\item Now, given our $G$-invariant and coherent system of metrics, we want to construct a global CAT(0) metric on $Y$ (or rather on the set of vertices of $Y$) which is invariant under the action of $G$. First, we endow each prism $P$ of $Y$, which is a product of cliques, with the $\ell^2$-product $\delta_P$ of the CAT(0) metrics defined on the corresponding cliques. Next, given two vertices $x,y \in Y$, a \emph{chain} $\Sigma$ between $x$ and $y$ is a sequence of vertices $$x_0=x, x_1, \ldots, x_{n-1}, x_n=y$$ such that, for every $0 \leq i \leq n-1$, there exists a prism $P_i$ containing both $x_i$ and $x_{i+1}$. Its \emph{length} is $$\ell(\Sigma)= \sum\limits_{i=1}^{n-1} \delta_{P_i}(x_i,x_{i+1}).$$ Finally, the global metric we define on (the vertices of) $Y$ is $$\delta^2 : (x,y) \mapsto \inf \{ \ell(\Sigma) \mid \text{$\Sigma$ chain between $x$ and $y$} \}.$$ It can be proved that the space $(Y,\delta^2)$ is indeed CAT(0) \cite[Proposition~3.11]{Qm}. 
\end{itemize}
Of course, the previous strategy can be adapted to other kinds of spaces.

\paragraph{Applications.} General criteria proved in \cite{Qm} include:
\begin{itemize}
	\item finding properly discontinuous actions on CAT(0) cube complexes \cite[Proposition 7.4]{Qm};
	\item finding virtually special and geometric actions on CAT(0) cube complexes \cite[Proposition 7.5]{Qm};
	\item finding geometric actions on CAT(0) spaces \cite[Theorem 7.7]{Qm}.
\end{itemize}
As corollaries, one can prove that:
\begin{itemize}
	\item graph products (along finite graphs) of CAT(0) groups are CAT(0) \cite[Theorem 8.20]{Qm};
	\item graph products (along finite graphs) of groups acting geometrically and virtually specially on CAT(0) cube complexes are virtually special \cite[Theorem 8.17]{Qm};
	\item acting properly on some CAT(0) cube complex is stable under wreath products \cite[Corollary 9.29]{Qm} (see also \cite{CornulierCommensurated, aTmW}).
\end{itemize}

\addcontentsline{toc}{section}{References}

\bibliographystyle{alpha}
\bibliography{QMintro}

\begin{thebibliography}{CCHO14}

\bibitem[Ago13]{MR3104553}
Ian Agol.
\newblock The virtual {H}aken conjecture.
\newblock {\em Doc. Math.}, 18:1045--1087, 2013.
\newblock With an appendix by Agol, Daniel Groves, and Jason Manning.

\bibitem[Ban84]{retracthypercube}
H.~J. Bandelt.
\newblock Retracts of hypercubes.
\newblock {\em Journal of Graph Theory}, 8(4):501--510, 1984.

\bibitem[BH99]{MR1744486}
Martin~R. Bridson and Andr{\'e} Haefliger.
\newblock {\em Metric spaces of non-positive curvature}, volume 319 of {\em
  Grundlehren der Mathematischen Wissenschaften [Fundamental Principles of
  Mathematical Sciences]}.
\newblock Springer-Verlag, Berlin, 1999.

\bibitem[BMW94]{quasimedian}
H.-J. Bandelt, H.M. Mulder, and E.~Wilkeit.
\newblock Quasi-median graphs and algebras.
\newblock {\em J. Graph Theory}, 18(7):681--703, 1994.

\bibitem[CCHO14]{weaklymoduloar}
J.~Chalopin, V.~Chepoi, H.~Hirai, and D.~Osajda.
\newblock Weakly modular graphs and nonpositive curvature.
\newblock {\em arXiv:1409.3892}, 2014.

\bibitem[Che00]{mediangraphs}
V.~Chepoi.
\newblock Graphs of some {$\rm CAT(0)$} complexes.
\newblock {\em Adv. in Appl. Math.}, 24(2):125--179, 2000.

\bibitem[Cor13]{CornulierCommensurated}
Y.~Cornulier.
\newblock Group actions with commensurated subsets, wallings and cubings.
\newblock {\em arxiv:1302.5982}, 2013.

\bibitem[CS11]{MR2827012}
Pierre-Emmanuel Caprace and Michah Sageev.
\newblock Rank rigidity for {CAT}(0) cube complexes.
\newblock {\em Geom. Funct. Anal.}, 21(4):851--891, 2011.

\bibitem[Far05]{FarleyPicture}
D.~Farley.
\newblock Actions of picture groups on \rm{CAT}(0) cubical complexes.
\newblock {\em Geometriae Dedicata}, 110(1):221--242, 2005.

\bibitem[Gen17a]{Qm}
A.~Genevois.
\newblock Cubical-like geometry of quasi-median graphs and applications to
  geometric group theory.
\newblock {\em PhD thesis}, 2017.

\bibitem[Gen17b]{MoiV}
A.~Genevois.
\newblock Embeddings in {T}hompson's groups from quasi-median geometry.
\newblock {\em arxiv:???}, 2017.

\bibitem[Gen17c]{aTmW}
A.~Genevois.
\newblock Lampligther groups, median spaces, and a-{T}-menability.
\newblock {\em arXiv:1705.00834}, 2017.

\bibitem[Gre90]{GreenGP}
E.~Green.
\newblock Graph products of groups.
\newblock {\em PhD Thesis}, 1990.

\bibitem[Gro87]{Gromov1987}
M.~Gromov.
\newblock {\em Hyperbolic Groups}, pages 75--263.
\newblock Springer New York, New York, NY, 1987.

\bibitem[GS97]{MR1396957}
Victor Guba and Mark Sapir.
\newblock Diagram groups.
\newblock {\em Mem. Amer. Math. Soc.}, 130(620):viii+117, 1997.

\bibitem[GS99]{MR1725439}
V.~S. Guba and M.~V. Sapir.
\newblock On subgroups of the {R}. {T}hompson group {$F$} and other diagram
  groups.
\newblock {\em Mat. Sb.}, 190(8):3--60, 1999.

\bibitem[GS06]{MR2193191}
V.~S. Guba and M.~V. Sapir.
\newblock Diagram groups are totally orderable.
\newblock {\em J. Pure Appl. Algebra}, 205(1):48--73, 2006.

\bibitem[KK13]{embeddingRAAG}
S.-H. Kim and T.~Koberda.
\newblock Embedability between right-angled {A}rtin groups.
\newblock {\em Geom. {T}opol.}, 17(1):493--530, 2013.

\bibitem[Mul80]{Mulder}
H.M. Mulder.
\newblock {\em The interval function of a graph}, volume 132 of {\em Math.
  Centre Tracts}.
\newblock Mathematisch Centrum, Amsterdam, 1980.

\bibitem[Neb71]{NebeskyMedian}
L.~Nebesk\'y.
\newblock Median graphs.
\newblock {\em Commentationes Mathematicae Universitatis Carolinae},
  12(2):317--325, 1971.

\bibitem[NR98]{NibloReeves}
G.~Niblo and L.~Reeves.
\newblock The geometry of cube complexes and the complexity of their
  fundamental groups.
\newblock {\em Topology}, 37(3):621--633, 1998.

\bibitem[Rol98]{Roller}
M.~Roller.
\newblock Pocsets, median algebras and group actions: {A}n extended study of
  {D}unwoody's construction and {S}ageev's theorem.
\newblock {\em dissertation}, 1998.

\bibitem[Sag95]{MR1347406}
Michah Sageev.
\newblock Ends of group pairs and non-positively curved cube complexes.
\newblock {\em Proc. London Math. Soc. (3)}, 71(3):585--617, 1995.

\bibitem[SW05]{alternative}
M.~Sageev and D.~Wise.
\newblock The {T}its alternative for {CAT}(0) cubical complexes.
\newblock {\em Bulletin of the London Mathematical Society}, 37(5):706--720,
  2005.

\bibitem[Wil92]{WilkeitRetractsHamming}
Elke Wilkeit.
\newblock The retracts of hamming graphs.
\newblock {\em Discrete Mathematics}, 102(2):197--218, 1992.

\end{thebibliography}

\end{document}